\documentclass{amsart}
\usepackage{amsfonts}
\usepackage{amsmath,amsthm}
\usepackage{amsfonts,amssymb}

\usepackage{enumerate}

\newtheorem{thm}{Theorem}[section]

\newtheorem{lem}[thm]{Lemma}

\theoremstyle{remark}
\newtheorem{rem}[thm]{Remark}

\numberwithin{equation}{section}

\def\({\Bigl(}
\def\){\Bigr)}

\newcommand{\al}{\alpha}

\def\vz{\varepsilon}
\def\oz{\omega}
\def\lz{\lambda}
\def\Lz{\Lambda}

\def\dz{\delta}

\def\gz{\gamma}

\def\sz{\sigma}

\def\iz{\infty}

\def\EE{\mathbb{E}}
\def\PP{\mathbb{P}}

\def\f{\frac}
\def\({\Bigl(}
\def \){ \Bigr)}

\def\vr{\varphi}

\def\bi{\bibitem}

 \def\d{{\delta}}
 
 \def\s{{\sigma}}

 \def\rb{{\mathbf r}}

 \def\PP{{\mathbb P}}
 
 \def\RR{{\mathbb R}}

        \def\supp{\operatorname{supp}}

\def\sa{\sigma}
\def\sz{\sigma}

\def\R{\mathbb{R}}

\def\lb{\langle}
\def\rb{\rangle}

\def\bi{\bibitem}

\begin{document}
\def\RR{\mathbb{R}}
\def\Exp{\text{Exp}}
\def\FF{\mathcal{F}_\al}
\def\ss{\mathbb{B}^d}

\title[] {Probabilistic and average
linear widths of  weighted Sobolev spaces on the ball  equipped with
a Gaussian measure}

\author{ Heping Wang} \address{School of  Mathematical Sciences,  Capital Normal
University
\\Beijing 100048,
 China}
\email{ wanghp@cnu.edu.cn.}

%\date{\today}

\thanks{Supported by
the National Natural Science Foundation of China (Project No.
11271263),
 the  Beijing Natural Science Foundation (1132001), and BCMIIS.}

\begin{abstract}Let
$L_{q,\mu}$, $1\leq q\leq\infty$, denotes the weighted $L_q$ space
of functions on the unit ball $\ss$ with respect to weight
$(1-\|x\|_2^2)^{\mu-\frac12},\,\mu\ge 0$, and let $W_{2,\mu}^r$ be
the weighted Sobolev space on $\ss$ with a Gaussian measure $\nu$.
We investigate  the probabilistic   linear $(n,\dz)$-widths
$\lz_{n,\dz}(W_{2,\mu}^r,\nu,L_{q,\mu})$ and
 the $p$-average  linear $n$-widths
   $\lz_n^{(a)}(W_{2,\mu}^r,\mu,L_{q,\mu})_p$,  and obtain their  asymptotic orders  for all $1\le q\le \infty$ and $0<p<\infty$.
\end{abstract}

\maketitle
\input amssym.def

\section{Introduction}

Let $K$ be a bounded subset of a normed  linear space $X$ with norm
$\|\cdot\|_X$. The  linear and Kolmogorov $n$-widths of the set $K$
in $X$ are defined by $$ \lz_n(K,X):= \inf\limits_{T_n}\sup_{x\in
K}\|x-L_nx\|_X,$$and
$$d_n(K,X):=\inf_{F_n}\sup_{x\in K}\inf_{y\in
F_n}\|x-y\|_X,$$ respectively, where $T_n$ runs over all linear
operators from $X$ to $X$ with rank at most $n$, and $F_n$ runs
through all possible linear subspaces of $X$ with dimension at most
$n$. They reflect the optimal errors of the ``worst"  elements of
$K$ in the approximation by linear operators with rank $n$ and
$n$-dimensional subspaces.

 Now let $W$ be a separable Banach space and assume that $W$ contains a Borel field $\mathcal {B}$
consisting of open subsets of $W$ and is equipped with a probability
measure $\gamma$ defined on $\mathcal {B}$. For $0<p<\infty$, the
$p$-average linear and Kolmogorov $n$-widths are defined by
 $$
\lz_n^{(a)}(W,\gamma,X)_p=\inf_{L_n}\Big(\int_W \|x-L_nx\|_X^p
\,\gamma (dx)\Big)^{1/p},$$ and
$$ d_n^{(a)}(W,\gamma,X)_p=\inf_{F_n}\Big(\int_W
\inf_{y\in F_n}\|x-y\|_X^p\,\gamma (dx)\Big)^{1/p},$$ respectively.
They reflect the optimal approximation of  ``most" elements of
classes by linear operators with rank $n$ and $n$-dimensional
subspaces. We stress that for  a centered Gaussian measure, the
averaging parameter $p$ is irrelevant   up to a constant (see
\cite[Theorem 1.2]{Fe} or  \cite[Corollary 1]{WZ}).

Let $\dz\in [0,1)$. The  probabilistic linear and Kolmogorov
$(n,\delta)$-widths of a set $W$ with a measure $\gamma$ in the
space $X$ are defined by
$$ \lz_{n,\dz}(W,\gamma,X)=\inf_{G_\dz} \lz_n(W\backslash
G_\dz,X),$$and $$ d_{n,\dz}(W,\gamma,X)=\inf_{G_\dz} d_n(W\backslash
G_\dz,X),$$respectively, where $G_\dz$ runs through all possible
measurable subsets in $W$ with measure $\gamma(G_\dz)\le \dz$.
Hence, the probabilistic linear and Kolmogorov  $(n,\dz)$-widths can
be understood as the $\gamma$-distribution of the approximation on
all measurable subsets of $W$ by  linear operators with rank $n$ and
by $n$-dimensional subspaces. Therefore, the probabilistic case
setting reflects the intrinsic structure of the class, and compared
with the worst case setting, allows one to give deeper analysis of
the
 approximation for the function class.

This paper is devoted to discussing the average and probabilistic
linear widths of  weighted Sobolev spaces on the unit ball with a
Gaussian measure.
 Let $\ss=\{x\in\mathbb{R}^d:\    \  \|x\|_2\le 1\}$ denote the unit
ball in $\mathbb{R}^d$,
 where $x\cdot y$ is
the usual inner product, and $\|x\|_2=(x\cdot x)^{1/2}$ is the usual
Euclidean norm.   For the weight $W_\mu(x)=(1-\|x\|_2^2)^{\mu-1/2}\
(\mu\ge 0)$, denote by $L_{p,\mu}\equiv L_p(\ss, W_\mu(x)\,dx),\
1\le p<\infty$, the space of measurable functions defined on $\ss$
with the finite norm $$
\|f\|_{p,\mu}:=\Big(\int_{\ss}|f(x)|^p\,W_\mu(x)dx\Big)^{1/p}, \ \ \
1\le p<\infty,$$ and for $p=\infty$ we assume that $L_{\infty,\mu}$
is replaced by the space $C(\ss)$ of continuous functions on $\ss$
with the uniform norm.

We denote by $\Pi_n^d$ the space of all polynomials in $d$ variables
of degree at most $n$, and by $\mathcal{V}_n^d$ the space of all
polynomials of degree $n$ which are orthogonal to polynomials of low
degree in $L_{2,\mu}$. Note that
$$ a_n^d:={\rm dim}\,\mathcal{V}_n^d=\binom {n+d-1} n\asymp n^{d-1}.$$ It is
well known (see \cite[p. 38 or p. 229]{DX}) that the spaces
$\mathcal{V}_n^d$ are just the eigenspaces corresponding to the
eigenvalues $-n(n+2\mu+d-1)=:-\lz_n$ of the second-order
differential operator
$$D_\mu:=\triangle-(x\cdot\nabla)^2-(2\mu+d-1)\,x\cdot \nabla,$$
where the $\triangle$ and $\nabla$ are the Laplace operator and
gradient operator respectively. More precisely,
$$D_\mu P=-n(n+2\mu+d-1)P=-\lz_nP\ \ {\rm for} \ P\in \mathcal{V}_n^d.$$
Also, the spaces $\mathcal{V}_n^d$  are mutually orthogonal in
$L_{2,\mu}$ and
\begin{equation}\label{1.1}
 L_{2,\mu}= \bigoplus_{n=0}^\infty
\mathcal{V}_n^d, \quad \ \ \ \  \Pi_n^d= \bigoplus_{k=0}^n
\mathcal{V}_n^d . \end{equation} Let $$\{\phi_{nk}\equiv
\phi_{nk}^{d}\ |\ k=1,\dots, a_n^d\}$$ be a fixed orthonormal basis
for $\mathcal{V}_n^d$. Then  we know that
$$\{\phi_{nk}\ |\ k=1,\dots, a_n^d,\ n=0,1,2,\dots\}$$ is
an orthonormal basis for $L_{2,\mu}$ with inner product $$ \lb
f,g\rb :=\int_{\ss} f(x) g(x)W_\mu(x)\, dx.$$ Denote by $S_n$ the
orthogonal projector  of $L_{2,\mu}$ onto $\Pi_n^d$, which is called
the Fourier partial summation operator. Evidently, for any $ f\in
L_{2,\mu}$, \eqref{1.1} can be rewritten in the form
\begin{equation*}\label{1.2}f=\sum_{n=0}^\infty Proj_n f, \ \ \ \ \ S_n(f):=\sum_{k=0}^n Proj_k
f, \end{equation*}
 where $Proj_n$ is the
orthogonal  projector from $L_{2,\mu}$ onto $\mathcal{V}_n^d$ and
can be written as
\begin{equation}\label{1.3}Proj_n(f)(x)={\sum_{k=1}^{a_n^d}} \lb
\phi_{nk},f\rb\phi_{nk}(x) =\int_{\ss} f(y) P_n (x,y)W_\mu(y) \,
dy,\end{equation} where $P_n(x,y)=
\sum_{k=1}^{a_n^d}\phi_{nk}(x)\phi_{nk}(y)$ is the reproducing
kernel for $\mathcal{V}_n^d$.
 It is  known
that for $\mu>0$, the kernel $P_n(x,y)$ has the compact
representation (see \cite{Xu1})
\begin{equation}\label{1.4}P_n(x,y)=b_d^\mu b_1^{\mu-\frac 12}\frac
{n+\lz}{\lz}\int_{-1}^1C_n^\lz
\Big((x,y)+u\sqrt{1-\|x\|_2^2}\sqrt{1-\|y\|_2^2}\Big)(1-u^2)^{\mu-1}du.\end{equation}
 Here, $C_n^\lz$
is the $n$-th degree Gegenbauer polynomial, $\lz=\mu+\frac{d-1}2$,\\
$b_d^\gamma:=(\int_{\ss}(1-\|x\|_2^2)^{\gamma-1/2}dx)^{-1}$. See
\cite{Xu1} for  the proof of the formula of $P_n(x,y)$, including
the limiting case of $\mu = 0$.

  Given $r\in \Bbb R$, we define the fractional power $(-D_\mu^d)^{r/2}$ of
the operator $-D_\mu^d$
 on $f$  by
$$(-D_\mu^d)^{r/2} (f)= \sum_{k=1}^\infty (k(k+2\mu+d-1))^{r/2} Proj_k(f)=\sum_{k=1}^\infty \lz_k^{r/2} Proj_k(f), $$
in the sense of distribution.
 We call
$f^{(r)}:=(-D_\mu^d)^{r/2}(f)$ the  $r$-th-order derivative of the
distribution $f$.

 For $r>0$, the
weighted Sobolev space $W_{2,\mu}^r\equiv W_{2,\mu}^r(\ss)$  is
defined by \begin{align*} W_{2,\mu}^r:=\Big\{&f=\sum_{n=1}^{\infty}
Proj_n(f)=\sum_{n=1}^{\infty}{\sum_{k=1}^{a_n^d}} \lb
\phi_{nk},f\rb\phi_{nk}\ \Big|\ \int_{\ss}f(x)W_\mu(x)dx=0,\\
& \|f\|_{W_{2,\mu}^r}^2:=\lb
f^{(r)},f^{(r)}\rb=\sum_{n=1}^{\infty}\lz_n^r
 \|Proj_nf\|_{2,\mu}^2=
\sum_{n=1}^{\infty}\lz_n^r \sum_{k=1}^{a_n^d} |
\hat{f}_{nk}|^2<\infty \Big\}\end{align*} with inner product
$$\lb f,g\rb _r:=\lb f^{(r)},g^{(r)}\rb.$$ Obviously, it
is a Hilbert space. If $1\le q\le \infty,\
r>(d+2\mu)(\frac{1}{2}-\frac{1}{q})_+$, then the space $W_{2,\mu}^r$
can be continuously embedded into the space $L_{q,\mu}$ (see
\cite[Lemma 1]{WZ}).

We equip $W_{2,\mu}^r $ with a Gaussian measure $\nu$ whose mean is
zero and whose correlation operator $C_\nu$ has eigenfunctions
$\phi_{lk},\ k=1,\dots, a_l^d,\, l=1,2,\dots$ and eigenvalues
$$\nu_l=\lz_l^{-s/2} ,\,\ \
s>d,$$ that is, $$C_\nu \phi_{lk}=\lambda_l^{-s/2} \phi_{lk}, \ \ \
\ \, k=1,\dots, a_l^d,\ \ l=1,2,\dots.\ $$ Then (see \cite[pp.
48-49]{Bo}),
$$\lb C_\nu f,g\rb _r=\int_{W_{2,\mu}^{r}}  \lb f,h\rb_r \lb g,h\rb_r
\nu (dh).$$

Denote by $\big(W_{2,\mu}^r\big)^* $  the space of all continuous
 linear functionals on $W_{2,\mu}^r$, and by
 $L_2(W_{2,\mu}^r,\nu)$ the usual  space of $\nu$-measurable functionals $\phi$ on $W_2^r$ with finite
 norm
 $$\|\phi\|_{L_2(\nu)}:=\Big(\int_{W_{2,\mu}^r}|\phi(x)|^2\nu(dx)\Big)^{1/2}.$$   Then $\big(W_{2,\mu}^r\big)^*=
 W_{2,\mu}^r$ can be
 embedded
 into  $L_2(W_{2,\mu}^r,\nu)$.
 Put
 (see \cite[p. 44]{Bo})
$$H(\nu)=\Big\{g \in W_{2,\mu}^r:\,\,|g|_{H(\nu)}:=\sup_{f\in \big(W_{2,\mu}^r\big)^*
=W_{2,\mu}^r,\,R_\nu(f)(f)\leq 1}|\langle f,g \rangle|<\infty\,
\Big\},$$
 where
$$R_{\nu}(f)(g):=\int_{W_{2,\mu}^r}\langle h,f \rangle_r\, {\langle
 h,g
 \rangle}_r\,\nu(dh),\qquad\quad f\,,g\in \big(W_{2,\mu}^r\big)^*=
  W_{2,\mu}^r.$$
  The space  $H(\nu)$ is called the  {\bf Cameron-Martin space } (or the reproducing kernel Hilbert space) of $\nu$.
  Set $\rho=r+s/2$.
 It is easy to see from \cite[pp. 48-49]{Bo} that the Cameron-Martin  space
 $H(\nu)$ of the Gaussian measure
$\nu$ is $W_{2,\mu}^\rho$, i.e.,
\begin{equation}\label{2.15}H(\nu)=W_{2,\mu}^{\rho}\ \ \ {\rm and}\ \ \ (\cdot,\cdot)_{H(\nu)}
=\lb\cdot,\cdot\rb_{\rho}.\end{equation}Then the covariance of $\nu$
\begin{equation}\label{2.16}R_\nu(f)(g)=\int_{W_{2,\mu}^r}
\lb f,h\rb_r \lb g,h\rb_r \nu (dh)=\lb C_\nu f,g\rb _r=\lb C_\nu
f,C_\nu g\rb _\rho.  \end{equation}

For any fixed $f_1,\dots, f_n\in W_{2,\mu}^r$,  the random vector
 $(\lb f,f_1\rb_r,\dots,\lb f,f_n\rb_r)$  on the measurable
space $(W_{2,\mu}^r,\nu)$ has the centered Gaussian distribution
with covariance matrix  $\big(\lb C_\nu
f_i,f_j\rb_r\big)_{i,j=1,\dots,n}=\big(\lb C_\nu f_i,C_\nu
f_j\rb_\rho\big)_{i,j=1,\dots,n}$. In special, on the cylindrical
subsets, the measure $\nu$ has the form: let $g_{k}$,
$k=1,2,\dots,n$ be any orthonormal system of functions in
$L_{2,\mu}$, $g_k\in H(\nu)=W_{2,\mu}^\rho$,
 and let $D$ be an arbitrary Borel subset
of $\Bbb R^{n}$, then the measure of the cylindrical subset
$$G=\Big\{f\in W_{2,\mu}^r\,|\ (\lb f,g_{1}^{(\rho)}\rb,\cdots,\lb
f,g_{n}^{(\rho)}\rb)\in D\Big\}$$  is equal to
$$\mu(G)=(2\pi)^{-\frac{n}{2}}\int_D
\exp(-\frac{1}{2}\sum_{l=1}^n u_l^2)du_1\cdots du_n .$$ See
\cite{Bo} and \cite{Kuo} for more information about Gaussian measure
on Banach spaces. Throughout  the paper, we always suppose that
$\nu$ is the above Gaussian measure on $W_{2,\mu}^r$,
$r>(d+2\mu)(\frac{1}{2}-\frac{1}{q})_+,\ s>d$, and $\rho=r+s/2$.

The aim of the paper is to study the average and probabilistic
linear widths of  $W_{2,\mu}^r$  with the  measure $\nu$ in
$L_{q,\mu}$. Probabilistic and average widths has been studied only
recently
 and are closely related with some other different
problems,  such as $\vz$-complexity and the minimal error of
problems of function approximation and integration by using standard
or general linear information, in the probabilistic and average case
setting (see \cite{TWW}, \cite{Ri}). A few interesting results have
been obtained.
  These include results on
probabilistic and average Kolmogorov and linear widths  of a
 Sobolev space of single-variate periodic functions with a Gaussian
measure and the $C^r[0,1]$ space equipped with the $r$-fold Wiener
measure in the $L_q$ norm for  $1\le q\le\infty$ (see \cite{FY1,
FY2, M1, M2, M3, M4, M5, MW, SW1, Wc}),   of a space of multivariate
periodic Sobolev functions  equipped with a Gaussian measure in the
$L_q$ norm for $1<q<\infty$ (see \cite{CF1, CF2, WJZ}), and of a
Sobolev space $W_2^r({\Bbb M}^{d-1})$ with a Gaussian measure  on
compact two-point homogeneous spaces ${\Bbb M}^{d-1}$ for $1\le
q\le\infty$ (see \cite{Wh}).  In the worst and average case setting,
the orders of the  Kolmogorov and linear widths of
 weighted Sobolev classes  on $\ss$ in $L_{q,\mu}$ were presented
 in \cite{WH} and \cite{WZ} respectively.
  More information about  average and probabilistic case setting
results can be found in   \cite{Ri} and \cite{TWW}.

 In this paper, we investigate probabilistic  and
average  linear widths of the weighted Sobolev space $W_{2,\mu}^r$
 with the
 measure $\nu$
in $L_{q,\mu}$,  and obtain the sharp estimation. Our main results
(Theorems \ref{thm1.1}
 and \ref{thm1.2} below) can be formulated  as follows:

 \begin{thm}\label{thm1.1}\ \ Let $\dz\in(0,1/2],\
  1\le q\le\infty,$ and let $\rho>d/2+2\mu d(1/2-1/q)_+$.   Then
    \begin{equation}\label{1.6}\lz_{n,\dz}(W_{2,\mu}^r,\nu,L_{q,\mu})\asymp \bigg\{\begin{array}{ll}
n^{-\frac{\rho}{d}+\frac12}\big(1+n^{-\min \{\frac12,\frac
1q\}}(\ln(\frac 1\dz))^{\frac12}\big)
,  &1\le q<\infty,\\
 n^{-\frac{\rho}{d}+\frac12}\big(\ln(n/\dz)\big)^{1/2},
&q=\infty,\end{array}
  \end{equation}where $A(n,\delta)\asymp B(n,\dz)$ means $A(n,\dz)\ll B(n,\dz)$
 and $B(n,\dz)\ll A(n,\dz)$,  $A(n,\dz)\ll
B(n,\dz)$ means that there exists a positive constant $c$
independent of $n$ and $\dz$ such that $A(n,\dz) \le cB(n,\dz)$ for
any $n\in\Bbb N$ and $\delta\in (0,1/2]$.\end{thm}

 \begin{thm}\label{thm1.2} Let $0< p<\infty,\ 1\leq q
 \le\infty$, and let $\rho$ be given as in Theorem \ref{thm1.1}.
  Then
\begin{equation}\label{1.8}\lz_n^{(a)}(W_{2,\mu}^r,\nu,L_{q,\mu})_p\asymp \bigg\{\begin{array}{ll}
n^{-\rho/d+1/2},\ \ \ \ &1\le q<\infty,\\
n^{-\rho/d+1/2}\sqrt{\ln (en)},
&q=\infty.\end{array}\end{equation}\end{thm}

\

Theorem \ref{thm1.2} extends (2.17) in \cite{WZ} which gave the
orders of  $\lz_n^{(a)}(W_{2,\mu}^r,\nu,L_{q,\mu})_p$ for $1\le q<
2+1/\mu$ and $0<p<\infty$. The proof of Theorem \ref{thm1.2} is
different from the one in \cite{WZ}. We use  Theorem \ref{thm1.1} to
prove Theorem \ref{thm1.2}. In order to prove Theorem \ref{thm1.1},
we also use the discretization method as in the previous works
concerning probabilistic widths. However, unlike the periodic case and the
Wiener measure case, we cannot find  an orthogonal transform between
a finite-dimensional function space and a sequence space of function
values  with comparable $L_q$ and $\ell_q$ norm, so we cannot use the invariant properties  of the standard Gaussian measure under
action of an orthogonal transform to give the
discretization theorems. Instead, we use the
comparison theorems for Gaussian measures. Unlike the spherical case, we cannot find a
equivalent-weight Marcinkiewicz-Zygmund (MZ) inequalities on the
ball (see Remark 3 in \cite{WH}), so we cannot use  known results on  the probabilistic
linear widths  of the identity matrix
on $\Bbb R^m$. Instead, we need to find out and use the probabilistic
linear widths of diagonal matrixes on $\Bbb R^m$.

\begin{rem}\ \ Let $BH(\nu)$ be the unit ball of  the Cameron-Martin space  of the Gaussian measure $\nu$.
 Then $BH(\nu)=BW_{2,\mu}^\rho$. It follows from \cite{WH}  that the classical
Kolmogorov and linear widths of $BH(\nu)$ in $L_{q,\mu}$ have the
same error order for $1\le q\le 2$, however, for $q>2$, the
 Kolmogorov   width $d_n(BH(\nu),L_{q,\mu})$ is
essentially less than the linear width $d_n(BH(\nu),L_{q,\mu})$, and
optimal linear operators lose to optimal nonlinear operators by a
factor $cn^{1/2-1/q}$ in the worst case setting. From
\cite[(2.16)]{WZ} and \eqref{1.8}, we know that in the average case
setting, the Kolmogorov and linear widths of $W_{2,\mu}^r$ in
$L_{q,\mu}$ have the same error order for $1\le q<\infty$. This
means that for ``most" functions in $W_{2,\mu}^r$, asymptotic
optimal linear operators are (modulo a constant) as good as optimal
nonlinear operators in the $L_{q,\mu} \ (1\le q<\infty)$
metric.\end{rem}

\begin{rem}\ \ Comparing \eqref{1.8} with (2.15) in \cite{WH},  we obtain that  in the average case setting,
 the Fourier partial summation operators $S_n$ are the asymptotically optimal linear operators in the
 $L_{q,\mu}$ metric if and only if  $1\le q<2+1/\mu$. This is completely different from
  the one-dimensional periodic case  and the spherical
  case, where the Fourier partial summation operators  are the asymptotically optimal linear operators in the
 $L_{q}$ metric for all $q,\ 1\le q\le \infty$ (see \cite{WZZ2}, \cite{WZZ1}, and \cite{Wh}).
\end{rem}

The paper is organized as follows. Section 2 contains some lemmas
concerning  probabilistic linear widths of diagonal matrixes in
$\Bbb R^m$ with the standard Gaussian measure. In Section 3,  we
give discretization theorems of estimates of probabilistic linear
widths. Finally, based on the results obtained in Sections 3 and 2,
we prove Theorems \ref{thm1.1} and \ref{thm1.2} in Section 4.

\section{Probabilistic linear widths of diagonal matrixes on $\Bbb R^m$ }\

Let $\ell_{q}^m\ (1\le q\le\infty)$ denote the $m$-dimensional
normed space of vectors  $x=(x_1,\dots,x_m)\in\Bbb R^m$, with the
$\ell_{q}^m$ norm
$$\|x\|_{\ell_{q}^m}:=\bigg\{\begin{array}{ll}\big(\sum_{i=1}^m
|x_i|^q \big)^{\frac 1q},\ \ \ &1\le q<\infty,\\ \max_{1\le i\le m}
|x_i|, &q=\infty.\end{array}$$ As usual, we identify $\RR^m$ with
the space $\ell_2^m$,  use the notation $\lb x, y \rb$ to denote the
Euclidean inner product of $x, y\in\RR^m$, and write $\|\cdot\|_2$
instead of $ \|\cdot\|_{\ell_{2}^m}$.  Consider in $\Bbb R^m$ the
standard Gaussian measure $\gz_m$, which is given by
$$\gz_m(G)=(2\pi)^{-m/2}\int_G\exp(-\|x\|_2^2/2)\,dx,$$where $G$ is
any Borel subset in $\Bbb R^m$.

Let $1\le q\le \infty$, $1\le n<m$, and $\delta\in [0,1)$. Then the
probabilistic linear $(n,\dz)$-widths of  a linear mapping
$T:\RR^m\to \ell_q^m$  is defined by
$$\lz_{n,\dz}(T:\RR^m\to
\ell_q^m,\gz_m)=\inf_{G_\dz}\inf_{T_n}\sup_{x\in\Bbb
 R^m\setminus G_\dz}\|Tx-T_nx\|_{\ell_{q}^m},$$  where
 $G_\dz$ runs over all possible Borel subset of $
 \Bbb R^m$ with measure $\gz_m(G_\dz)\le \dz$,  and $T_n$  all linear operators from $\Bbb R^m$ to $\ell_q^m$ with rank
 at most
 $n$.

 There are a few results devoted to studying
  the  probabilistic widths of a linear mapping $T:\RR^m\to \ell_q^m$ (see \cite{CNL, DW, FY1, M3, M4,
  MW}). Throughout  the section, we use  the letter $D$ to
denote an   $m\times m$ real diagonal matrix $diag(d_1,\dots,d_m)$
with $d_1\ge d_2\ge\dots\ge d_m>0$, the letter   $D_n$ the diagonal
matrix $diag(d_1,\dots,d_n, 0,\dots,0)$ for $1\leq n\leq m$, and the
letter $I_m$ the $m$ by $m$ identity matrix. Moreover,
$\{e_1,\cdots, e_m\}$ denotes the standard orthonormal basis in
$\RR^m$:
 $$e_1=(1,0,\cdots, 0),\cdots, e_m=(0,\cdots, 0,1).$$
 The following lemmas will be used in the proof of Theorem
  \ref{thm1.1}.

  \begin{lem}\label{lem2.1} (1) {\rm (\cite{FY1}).} If $1\le q\le 2, \ m\ge 2n,  \delta\in
(0,1/2]$, then \begin{equation}\label{2.1}\lz_{n,\dz}(I_m:\RR^m\to
\ell_q^m,\gz_m)\asymp
m^{1/q}+m^{1/q-1/2}\sqrt{\ln(1/\dz)}.\end{equation}

 (2) {\rm (\cite{M4}).} If $2\le q<\infty,\  m\ge
2n,\ \delta\in (0,1/2]$, then
\begin{equation}\label{2.2}\lz_{n,\dz}(I_m:\RR^m\to
\ell_q^m,\gz_m)\asymp m^{1/q}+\sqrt{\ln(1/\dz)}.\end{equation}

(3) {\rm (\cite{MW}).} If $q=\infty,\  m\ge 2n,\ \delta\in (0,1/2]$,
then
\begin{equation}\label{2.3}\lz_{n,\dz}(I_m:\RR^m\to
\ell_q^m,\gz_m)\asymp \sqrt{\ln((m-n)/\dz)}\asymp \sqrt{\ln
m+\ln(1/\dz)}.\end{equation}\end{lem}

\begin{lem} \label{lem2.2} Assume that $$\sum_{i=1}^md_i^\beta\le
C(m,\beta) \ \ \ {\rm for \ \ some}\ \ \beta>0.$$ Then for  $2\le
q\le\infty,\ m\ge 2n,\ \delta\in (0,1/2]$, we have
\begin{equation}\label{2.4}\lz_{n,\dz}(D:\RR^m\to \ell_q^m,\gz_m)\ll \Big(\frac {C(m,\beta)}{n+1}\Big)^{\frac
1\beta}\cdot \Bigg\{\begin{array}{ll}(m^{1/q}+\sqrt{\ln(1/\dz)} ,
&2\le q<\infty, \\ \\ \sqrt{\ln m+\ln(1/\dz)}, &q=\infty.\end{array}
\end{equation}\end{lem}
\begin{proof} First we show that \begin{equation}\label{2.5}\int_{\Bbb
R^m}\|Dx-D_nx\|_{\ell_q^m}\,\gamma_m(dx)\ll  \Big(\frac
{C(m,\beta)}{n+1}\Big)^{\frac 1\beta}\cdot
\Bigg\{\begin{array}{ll}m^{1/q} , &2\le q<\infty,
\\ \\ \sqrt{\ln m}, &q=\infty.\end{array}\end{equation}Indeed, for $2\le q<\infty$, it follows from
\cite[(2.9)]{DW} that
\begin{equation}\label{2.6}\(\int_{\Bbb
R^m}\|Dx-D_nx\|_{\ell_q^m}^q\,\gamma_m(dx)\)^{1/q}=C(q)\(\sum_{k=n+1}^md_k^q\)^{1/q},
\end{equation}where
$C(q)=\Big(\pi^{-\frac{1}{2}}2^\frac{q}{2}\Gamma(\frac{q+1}{2})\Big)^{1/q}$.
Since
$${(n+1)}d_{n+1}^\beta\le \sum_{i=1}^{n+1}d_i^\beta\le \sum_{i=1}^md_i^\beta\le
C(m,\beta),$$we get
\begin{equation}\label{2.7}d_{n+1}\le \Big(\frac {C(m,\beta)}{n+1}\Big)^{\frac
1\beta}.\end{equation} Hence, we have\begin{align*}\int_{\Bbb
R^m}\|Dx-D_nx\|_{\ell_q^m}\,\gamma_m(dx)&\le \(\int_{\Bbb
R^m}\|Dx-D_nx\|_{\ell_q^m}^q\,\gamma_m(dx)\)^{1/q}\ll
\(\sum_{k=n+1}^md_k^q\)^{1/q}\\ & \le d_{n+1}(m-n)^{1/q}\ll m^{1/q}
\Big(\frac {C(m,\beta)}{n+1}\Big)^{\frac 1\beta},
\end{align*}proving \eqref{2.5} for $2\le q<\infty$. For $q=\infty$,
it follows from \eqref{2.6} and \eqref{2.7}  that for any
$1<q_1<\infty$,
\begin{align} \int_{\R^m}\|Dx-D_nx\|_{\ell_\infty^m}
\gz_m(dx)&\le
\Big(\int_{\R^m}\|Dx-Dx\|_{\ell_{q_1}^m}^{q_1}\gz_m(dx)\Big)^{1/q_1}\notag\\
&=C(q_1)\(\sum_{k=n+1}^md_k^{q_1}\)^{1/q_1}\le
C(q_1)d_{n+1}(m-n)^{1/q_1}\notag\\ &\le
c\,m^{\frac{1}{q_1}}\Big(\Gamma\big(\frac{q_1
+1}{2}\big)\Big)^{\frac{1}{q_1}} \Big(\frac
{C(m,\beta)}{n+1}\Big)^{\frac 1\beta}.\label{2.8}\end{align} By
Stirling's formula (see \cite[p. 18]{AAR}):
$$\lim_{x\to+\infty}\frac {\Gamma(x)}{\sqrt{2\pi}x^{x-\frac{1}{2}}\exp(-x)}
=1,$$ we  obtain
$$(\Gamma(\frac{x
+1}{2}))^\frac{1}{x}\leq
c\,(\sqrt{2\pi\,})^\frac{1}{x}(\frac{x+1}{2})
^\frac{1}{2}\exp(-\frac{x+1}{2x})\le c\,x^\frac{1}{2}.$$ Hence,
taking $q_1=\ln (e^2m)$, we obtain from (2.8) that
\begin{equation*}
\int_{\R^m}\|Dx-D_nx\|_{\ell_\infty^m} \gz_m(dx)\le
c\,m^{\frac{1}{q_1}}q_1^{\frac 12} \Big(\frac
{C(m,\beta)}{n+1}\Big)^{\frac 1\beta}\ll \sqrt{\ln m}\Big(\frac
{C(m,\beta)}{n+1}\Big)^{\frac 1\beta},
\end{equation*}
which completes the proof of \eqref{2.5}.

Now we shall show \eqref{2.4}. We need the following lemma.

\begin{lem}\label{lem3.3} (\cite[(1.7.7)]{Bo}, \cite[ p. 47]{Pis}) \
\ Let  $F:\R^m\to\R$ be a function  satisfying the following
Lipschitz condition
$$|F(x)-F(y)|\le \sz\|x-y\|_2,\ \ x,y\in\R^m,$$
for some  $\s>0$ independent of $x$ and $y$. If $X\sim N_m(0,
I_m)_{\RR^m}$ is an $\RR^m$-valued Gaussian random vector with mean
$0$ and covariance matrix $I_m$, then for all $t>0$
\begin{equation}\label{2.9}\PP(|F(X)-\EE F(X)|\ge
t)\le 2\exp(-\frac{t^2}{K^2\sz^2}),\end{equation} with $K>0$ being
an absolute constant. \eqref{2.9} is called the Gaussian
concentration inequality. \end{lem}

We continue to prove Lemma \ref{lem2.2}. Let
$F(x)=\|Dx-D_nx\|_{\ell_{q}^m}$ for $x\in\RR^m$ and $2\leq q \leq
\infty$. By \eqref{2.7} we obtain for any $x,y\in \RR^m$
\begin{align*}|F(x)-F(y)|&\le \|(D-D_n)(x-y)\|_{\ell_{q}^m}=\Big(\sum_{k=n+1}^m d_k^q|x_k-y_k|^q\Big)^{\f1q}\\
&\leq d_{n+1}\Big(\sum_{k=n+1}^m |x_k-y_k|^q\Big)^{\f1q} \le
d_{n+1}\Big(\sum_{k=n+1}^m |x_k-y_k|^2\Big)^{\f12}\\& \le \Big(\frac
{C(m,\beta)}{n+1}\Big)^{\frac 1\beta} \|x-y\|_2=:\sz_1 \|x-y\|_2,
\end{align*}
 where $\sa_1:=\Big(\frac
{C(m,\beta)}{n+1}\Big)^{\frac 1\beta}$. Thus,
 applying the Gaussian
concentration inequality \eqref{2.9} yields
\begin{equation*}\PP(|F(X)-\mathbb{E}F(X)|\ge
t)\le 2\exp(-\frac{t^2}{K^2\sz_1^2}),\   \  \forall
t>0,\end{equation*} where, here and in what follows, $X\sim
N_m(0,I_m)_{\RR^m}$. In particular, this implies that for
$Q_\delta=\big\{x\in \R^m:\
F(x)>\mathbb{E}F(X)+K\sigma_1\sqrt{\ln(2/\delta)}\big\} $ with
$\d\in (0,1)$,
$$\gamma_m(Q_\dz)\le \PP(|F(X)-\mathbb{E}F(X)\big|>K\sigma_1\sqrt{\ln(2/\delta)})\le \dz. $$
By the definition of the linear $(n,\d)$ widths, this last equation
further implies that
\begin{align}\lz_{n,\dz}&(D:\RR^m\to\ell_q^m,\gamma_m)\le \sup_{\R^m\setminus
Q_\dz}\|Dx-D_nx\|_q\notag\\
&\le
\mathbb{E}F(X)+K\sigma_1\sqrt{\ln(2/\delta)}.\label{2.10}\end{align}
 On the other hand, however, using \eqref{2.5}, we have
\begin{equation}\label{2.11}\mathbb{E}F(X)=\EE \|DX-D_nX\|_{\ell_q^m}=\int_{\Bbb
R^m}\|Dx-D_nx\|_{\ell_q^m}\,\gamma_m(dx)=:\s_2,\end{equation} where
$\s_2$ denotes the right expression of \eqref{2.5}. Thus, combining
\eqref{2.10} with \eqref{2.11}, we deduce the desired upper
estimates.

\end{proof}

\section{Discretization of the probabilistic linear widths}

This  section  is devoted to obtaining the discretization theorems
which give the reduction of the calculation of  probabilistic widths
of a given function class to the computation of probabilistic widths
of a finite-dimensional set equipped with the standard Gaussian
measure.

Let $\eta$ be a nonnegative $C^\infty$-function on $[0,+\infty)$
supported in $[0,2]$ and  equal to 1 on $[0,1]$. We also suppose
that $\eta(x)>0$ for $x\in [0,2)$. We define
\begin{equation}\label{3.1}L_{n,\eta}(x,y):=\sum_{j=0}^{\infty}\eta (\frac j n)P_j(x,y), \ \
 \ \ x,y\in \ss,\end{equation}where
 $P_k(x,y)$ is given in \eqref{1.4}. It is well known that for any $P\in\Pi_n^d$,
 \begin{equation}\label{3.1.1}\int_{B^d}L_{n,\eta}(x,y)P(y)W_\mu(y)dy=P(x),\ \ \ \ x\in
 \ss.\end{equation}
 We recall that the operator $S_n$ is the
 orthogonal projection operator from $L_{2,\mu}$ to $\Pi_n^d$,
 i.e., $$S_nf(x)=\sum_{k=0}^n
 Proj_k(f)(x)=\sum_{k=0}^n\int_{\ss}f(y)P_k(x,y)W_\mu(y)dy,$$where $Proj_k$ is the
orthogonal  projector from $L_{2,\mu}$ onto $\mathcal{V}_k^d$
defined by \eqref{1.3}. For any $f\in L_{2,\mu}$, we define
\begin{equation}\label{3.2}\dz_1(f)=S_{2}(f),\ \
\dz_k(f)=S_{2^{k}}(f)-S_{2^{k-1}}(f) \ \ {\rm for }\ \
k=2,3\dots.\end{equation}
 Then
for $x\in \ss$,
$$\dz_k(f)(x)=\sum_{j=2^{k-1}+1}^{2^k}Proj_j(f)(x)=\lb f, M_k(\cdot,x)\rb,$$
where
\begin{equation}\label{3.3}M_k(x,y)=\sum_{j=2^{k-1}+1}^{2^k}\sum_{i=1}^{a^d_j}\phi_{ji}(x)\phi_{ji}(y)
=\sum_{j=2^{k-1}+1}^{2^k}P_k(x,y) \end{equation} is the reproducing
kernel of the Hilbert space
$L_{2,\mu}\bigcap\big(\bigoplus\limits_{j=2^{k-1}+1}^{2^k}\mathcal{V}_j^d\big)$,
and $\{\phi_{j1},\dots, \phi_{ja_j^d}\}$ is an orthonormal basis for
$\mathcal{V}_j^d$. Obviously, for any $x\in \ss$,  $M_k(\cdot,x)\in
\bigoplus\limits_{j=2^{k-1}+1}^{2^k}\mathcal{V}_j^d$  and for any
$f\in \bigoplus\limits_{j=2^{k-1}+1}^{2^k}\mathcal{V}_j^d$,
$f(x)=\dz_k(f)(x)=\lb f, M_k(\cdot,x)\rb.$ In special, for any
$x,y\in\ss$,
\begin{equation*}\label{3.4}\lb
M_k(\cdot,x),M_k(\cdot,y)\rb=M_k(x,y).\end{equation*}

 Now we  introduce
the  metric $\tilde\rho$  on $\ss$:
$$ \tilde\rho(x,y):= \arccos \Big(
(x,y)+\sqrt{1-\|x\|_2^2}\sqrt{1-\|y\|_2^2}\Big), \ \ \ \ \
x,y\in\ss.$$For $r>0,\ x\in \ss$, we set
$B_{\tilde\rho}(x,r):=\{y\in \ss\ | \ \ \tilde\rho(x,y)\le r\}$.
 For $\delta>0,\ n\in \Bbb N$, we say that a
finite subset $\Lz\subset \ss$ is maximal $(\frac \delta n,
\tilde\rho)$-separated if
$$ \ss\subset \bigcup_{y\in \Lz}B_{\tilde\rho}(y,\frac \delta n)\ \ \ {\rm
and}\ \ \ \min_{y\neq y'\in \Lz}\tilde\rho(y,y')\ge \frac \delta
n.$$ The following cubature formulae and Marcinkiewicz-Zygmund (MZ)
inequalities  on  $\ss$  are crucial for establishing discretization
theorems.

\begin{lem}\label{lem4.1} (see \cite{Dai1, PX}) There exists a constant
$\gamma
>0$ depending only on $d$  and $\mu$ such that  for any $\dz\in
(0,\gamma]$, any positive integer $n$, and any maximal
$(\dz/n,\tilde\rho)$-separated subset $\Lz\subset \ss$, there exists
a sequence of positive numbers\\ $\omega_\xi\asymp
n^{-d}(\frac1n+\sqrt{1-\|\xi\|_2^2})^{2\mu}, \ \xi\in\Lz$, for which
the following quadrature formula holds for all $f \in \Pi_{4n}^d$,
$$\int_{\ss} f(x)W_\mu(x)\,dx=\sum_{\xi\in \Lz}\omega_\xi f(\xi).$$
Moreover,  for any $1\le q\le\iz,\ f\in\Pi_{n}^d$, we have
$$\|f\|_{q,\mu}\asymp
\bigg\{\begin{array}{ll}\big(\sum_{\xi\in \Lz}
 |f(\xi)|^q \,\omega_\xi\big)^{\frac 1q},\ \ \ &1\le q<\infty,\\
\max_{\xi\in \Lz} |f(\xi)|, &q=\infty,\end{array} $$ where the
constants of equivalence depend only on $d$ and $\mu$.\end{lem}

\begin{lem}\label{lem4.2} (see \cite[Lemma 3]{WH}) Let $\mu>0,\ \beta\in(0,1/(2\mu))$ and
let $\Lz,\ \oz_\xi$ be given as in Lemma \ref{lem4.1}. Then
$$\sum_{\xi \in \Lz} \oz_\xi^{-\beta}\ll n^{d(1+\beta)}.$$\end{lem}

For $k=1,2,\dots$, let $\gamma>0$ be the same as in Lemma
\ref{lem4.1} and
 let $\Lz_k=\{\xi_1,\dots,\xi_{u_k}\}$
be a maximal $(\gamma {2^{-(k+2)}},\tilde\rho)$-separated subset of
$\ss$. It is easy to know that  $u_k\asymp 2^{kd}$. By Lemma
\ref{lem4.1}, there exists a sequence of positive numbers $w_i
\asymp 2^{-kd}(2^{-k}+\sqrt{1-\|\xi_i\|_2^2})^{2\mu}, \ 1\le i\le
u_k$, for which the following quadrature formula holds for all $f
\in \Pi_{2^{k+4}}^d$,
\begin{equation}\label{3.5}\int_{\ss} f(x)W_\mu(x)\,dx=\sum_{i=1}^{u_k}w_i f(\xi_i). \end{equation}
Moreover,  for any $1\le q\le\iz,\ f\in\Pi_{2^{k+2}}^d$, we have
\begin{equation*}\label{3.6}\|f\|_{q}\asymp
\Big(\sum_{i=1}^{u_k}
|f(\xi_i)|^q\,\oz_i\Big)^{1/q}=\|U_k(f)\|_{\ell_{q,\oz}^{u_k}},\end{equation*}where
 $U_k:\Pi_{2^{k+2}}^d\longmapsto \Bbb R^{u_k}$ is defined by
\begin{equation}\label{3.7}U_k(f)=(f(\xi_1),\dots,f(\xi_{u_k})),\end{equation}and
for $x\in \Bbb R^{u_k}$,
$$\|x\|_{\ell_{q,\oz}^{u_k}}:=\bigg\{\begin{array}{ll}\big(\sum_{i=1}^{u_k} |x_i|^q \oz_i\big)^{\frac 1q},\ \ \ &1\le q<\infty,\\
\max_{1\le i\le u_k} |x_i|, &q=\infty.\end{array} $$ Next, we define
the operator $T_k:\Bbb R^{u_k}\longmapsto \Pi_{2^{k+1}}^d$ by
\begin{equation}\label{3.8}T_ka(x):=\sum_{i=1}^{u_k}a_iw_iL_{2^{k},\eta}(x,\xi_i),\end{equation}
where $a=(a_1,\dots,a_{u_k})\in \Bbb R^{u_k}$, $L_{n,\eta}$ is
defined as in \eqref{3.1}. It is shown in \cite[(2.15)]{WH} that for
any $q$, $1\le q\le \infty$,
\begin{equation}\label{3.9}\|T_ka\|_{q,\mu}\ll \|a\|_{\ell_{q,\oz}^{u_k}}.  \end{equation}
 For  $f\in \Pi_{2^{k}}^d$, by \eqref{3.1.1} and \eqref{3.5} we get
$$f(x)=\int_{\ss}f(y)L_{2^{k},\eta}(x,y)d\sz (y)=\sum_{j=1}^{u_k}w_jf(\xi_j)L_{2^{k},\eta}(x,\xi_j),$$
which means that $f=T_kU_kf$ for any $f\in \Pi_{2^{k}}^d$. We also
use the letters $S_k,\,R_k,\, V_k$ to denote    $u_k\times u_k$ real
diagonal matrixes  as follows:
\begin{equation}\label{3.10}S_k=diag(\oz_1^{\frac 12},\dots,\oz_1^{\frac 12}),\  \ R_k=diag(\oz_1^{\frac 1q},\dots,\oz_1^{\frac 1q}), \ \
V_k=diag(\oz_1^{-\frac 12+\frac 1q},\dots,\oz_1^{-\frac 12+\frac
1q}),
\end{equation}
 and use the letter $R_k^{-1}$ to represent the inverse
matrix of $R_k$. Clearly, for $x\in \Bbb R^{u_k}$,
$$\|R_kx\|_{\ell_q^{u_k}}=\|x\|_{\ell_{q,\oz}^{u_k}}\ \ \ {\rm and}\ \ \ R_k=V_kS_k.$$

\begin{lem}\label{lem3.3} For any $z=(z_1,\dots,z_{u_k})\in\Bbb R^{u_k}$, we have
\begin{equation}\label{3.11}\Big\|\sum_{j=1}^{u_k}\oz_j^{1/2}z_jM_k(\cdot,\xi_j)\Big\|_{2,\mu}\ll
\|z\|_{\ell_2^{u_k}},\end{equation} where $M_k(x,y)$ is given in
\eqref{3.3}, and $\{\xi_1,\dots,\xi_{u_k}\}$ is defined as
above.\end{lem}

\begin{proof}Denote by $K$ the set $\big\{g\in \bigoplus\limits_{j=2^{k-1}+1}^{2^k}\mathcal{V}_j^d\ \big|\ \|g\|_{2,\mu}\le1\big\}$. Since
$$\sum\limits_{j=1}^{u_k}\oz_j^{1/2}z_jM_k(\cdot,\xi_j)\in
L_{2,\mu}\bigcap\big(\bigoplus\limits_{j=2^{k-1}+1}^{2^k}\mathcal{V}_j^d\big),$$
By the Riesz representation theorem and Cauchy-Schwarz inequality we
have
\begin{align*}\Big\|\sum_{j=1}^{u_k}\oz_j^{1/2}z_jM_k(\cdot,\xi_j)\Big\|_{2,\mu}&=\sup_{g\in K}\Big|\big\langle
 \sum_{j=1}^{u_k}\oz_j^{1/2}z_jM_k(\cdot,\xi_j), g\big\rangle\Big|=\sup_{g\in K}\big|\sum_{j=1}^{u_k}\oz_j^{1/2}z_jg(\xi_j)\big|\\ &
\le\sup_{g\in K}
\Big(\sum_{j=1}^{u_k}|z_j|^2\Big)^{1/2}\Big(\sum_{j=1}^{u_k}|g(\xi_j)|^2\oz_j\Big)^{1/2}\\
&\ll \sup_{g\in K}(\sum_{j=1}^{u_k}|z_j|^2\Big)^{1/2}
\|g\|_{2,\mu}\le \|z\|_{\ell_2^{u_k}},\end{align*}which proves
\eqref{3.11}.\end{proof}

\begin{thm}\label{thm4.3} Let $1\le q\le\infty,\
\sz\in (0,1)$, and let the sequences of numbers $\{n_k\}$ and
$\{\sz_k\}$ be such that $0\le n_k\le u_k\asymp 2^{kd},\
\sum_{k=1}^\infty n_k\le n,\ \sz_k\in(0,1),\ \sum_{k=1}^\infty
\sz_k\le\dz$. Then
\begin{equation}\label{3.13}\lz_{n,\sz}(W_{2,\mu}^r,\nu,L_{q,\mu})\ll\sum_{k=1}^\infty
2^{-k\rho}\lz_{n_k,\sz_k}(V_k:\Bbb R^{u_k}\to
\ell_q^{u_k},\gz_{u_k}).\end{equation}  \end{thm}

\begin{proof} For convenience, we
write
$$\lz_{n_k,\sz_k}:=\lz_{n_k,\sz_k}(V_k:\Bbb R^{u_k}\to
\ell_q^{u_k},\gz_{u_k}),$$ where $\gz_{u_k}$ is the standard
Gaussian measure in $\Bbb R^{u_k}$. Denote by $L_k$ a linear
operator from $\Bbb R^{u_k}$ to $\Bbb R^{u_k}$ such that the rank of
$L_k$ is at most $n_k$ and
$$\gz_{u_k}\Big(\{y\in \Bbb  R^{u_k}\ |\
\|V_ky-L_ky\|_{\ell_q^{u_k}}>2\lz_{n_k,\sz_k}\}\Big)\le \sz_k.$$Then
for any $f\in W_{2,\mu}^r$, by \eqref{3.9} we have
\begin{align}\|\dz_k(f)-T_kR_k^{-1}L_kS_kU_k\dz_k(f)\|_{q,\mu}&=\|T_kU_k\dz_k(f)-T_kR_k^{-1}L_kS_kU_k\dz_k(f)
\|_{q,\mu}\notag \\ &\ll
\|U_k\dz_k(f)-R_k^{-1}L_kS_kU_k\dz_k(f)\|_{\ell_{q,\oz}^{u_k}}\notag\\
&=\|V_kS_kU_k\dz_k(f)-L_kS_kU_k\dz_k(f)\|_{\ell_q^{u_k}},
\label{3.14}\end{align}where $\dz_k,\, U_k,\, T_k$, and $ S_k,\,
V_k,\, R_k$ are defined by \eqref{3.2}, \eqref{3.7},  \eqref{3.8},
and \eqref{3.10}, respectively. Denote $y=S_kU_k\dz_k(f)
=(\oz_1^{1/2}\dz_k(f)(\xi_1),\dots,\oz_{u_k}^{1/2}\dz_k(f)(\xi_{u_k}))\in\Bbb
R^{u_k}$. Note that for $x\in \ss$,
$$\dz_k(f)(x)=\lb f, M_k(\cdot,x)\rb=\lb f^{(-r)}, M_k^{(-r,0)}(\cdot,x)\rb_r
=\lb f, M_k^{(-2r,0)}(\cdot,x)\rb_r,$$where $M_{k}^{(r_1,0)}(x,y)$
is the $r_1$-order partial derivative of $M_k(x,y)$ with respect to
the variable $x$, $r_1\in\Bbb R$. By the property of  Gaussian
measures  we know that  if a random vector $f$ in $W_{2,\mu}^r$ is a
centered Gaussian random vector with covariance operator $C_\nu$,
then the  vector
$$y=S_kU_k\dz_k(f)=(\lb f,
\oz_1^{1/2}M_k^{(-2r,0)}(\cdot,\xi_1)\rb_r,\dots, \lb f,
\oz_{u_k}^{1/2}M_k^{(-2r,0)}(\cdot,\xi_{u_k})\rb_r )$$ in $\Bbb
R^{u_k}$ is a random vector with a centered Gaussian distribution
$\gamma$ in $\Bbb R^{u_k}$ and its covariance matrix $C_\gamma$ is
given by
$$C_\gamma=\Big( \lb C_\nu (\oz_i^{1/2}M_k^{(-2r,0)}(\cdot,\xi_i)),
\oz_j^{1/2}M_k^{(-2r,0)}(\cdot,\xi_j)\rb_r\Big)_{i,j=1}^{u_k}.$$Note
that
\begin{align*} &\quad \ \Big\lb C_\nu (\oz_i^{1/2}M_k^{(-2r,0)}(\cdot,\xi_i)),
\oz_j^{1/2}M_k^{(-2r,0)}(\cdot,\xi_j)\Big\rb_r\\ &=\Big\lb
\oz_i^{1/2}M_k^{(-2r-s,0)}(\cdot,\xi_i), \oz_j^{1/2}M_k^{(-2r,0)}(\cdot,\xi_j)\Big\rb_r\\
&=\Big\lb \oz_i^{1/2}M_k^{(-\rho,0)}(\cdot,\xi_i),
\oz_j^{1/2}M_k^{(-\rho,0)}(\cdot,\xi_j)\Big\rb. \end{align*}Since
for any $z=(z_1,\dots,z_{u_k})\in \Bbb R^{u_k}$,
$$\sum_{j=1}^{u_k}\oz_j^{1/2}z_jM_k(\cdot,\xi_j)\in
\bigoplus_{j=2^{k-1}+1}^{2^k}\mathcal{V}_j^d,$$ by \eqref{3.11} we
get
\begin{align} \int_{\Bbb R^{u_k}}(y,z)^2\gamma
(dy)&=zC_\gamma z^T=\sum_{i,j=1}^{u_k}z_iz_j \big\lb
\oz_i^{1/2}M_k^{(-\rho,0)}(\cdot,\xi_i),\oz_j^{1/2} M_k^{(-\rho,0)}(\cdot,\xi_j)\big\rb\notag\\
&=\Big\lb
\sum_{j=1}^{u_k}\oz_j^{1/2}z_jM_k^{(-\rho,0)}(\cdot,\xi_j),
\sum_{j=1}^{u_k}\oz_j^{1/2}z_jM_k^{(-\rho,0)}(\cdot,\xi_j)\Big\rb\notag\\
&= \Big
\|\sum_{j=1}^{u_k}\oz_j^{1/2}z_jM_k^{(-\rho,0)}(\cdot,\xi_j)\Big\|_2^2\asymp
2^{-2k\rho}\Big \|\sum_{j=1}^{u_k}\oz_j^{1/2}z_jM_k(\cdot,\xi_j)\Big\|_2^2\notag\\
&\ll 2^{-2k\rho}\|z\|_{\ell_2^{u_k}}=2^{-2k\rho}\int_{\Bbb
R^{u_k}}(y,z)^2\gz_{u_k}(dy).\label{3.15}\end{align}

Now consider the subset of $W_{2,\mu}^r$ $$G_{k}:=\left\{f\in
W_{2,\mu}^r\ |\
\|\dz_k(f)-T_kR_k^{-1}L_kS_kU_k\dz_k(f)\|_q>2c_1c_22^{-k\rho}
\lz_{n_k,\sz_k}\right\},$$ where $c_1,c_2$ are the positive
constants given in \eqref{3.14} and \eqref{3.15}. Then it follows
from \eqref{3.14} that
\begin{align*} \nu(G_k)&\le \nu\Big(\Big\{f\in
W_{2,\mu}^r\ |\
\|V_kS_kU_k\dz_k(f)-L_kS_kU_k\dz_k(f)\|_{\ell_q^{u_k}}>2c_22^{-k\rho}
\lz_{n_k,\sz_k}\Big\}\Big)\\ &=\gamma\Big(\big\{y\in\Bbb R^{u_k}\ |\
\|V_ky-L_ky\|_{\ell_q^{u_k}}>2c_22^{-k\rho}
\lz_{n_k,\sz_k}\big\}\Big).\end{align*} By Theorem 1.8.9 in \cite[p.
29]{Bo}, we know that if $\gamma_1$ and $\gamma_2$ are two centered
Gaussian measures on $\Bbb R^N$ and satisfy
$$\int_{\Bbb R^N}(y,x)^2\gamma_1(dx)\ge \int_{\Bbb R^N}(y,x)^2\gamma_2(dx),
\ \ \ \forall\, y\in \Bbb R^N,$$then for every convex symmetric set
$E$, $\gamma_1(E)\le \gamma_2(E).$ Note that for any $t>0$, the set
$\{y\in \Bbb R^{u_k} |\ \|V_ky-L_ky\|_{\ell_q^{u_k}}\le t\}$ is
convex symmetric. It follows from \eqref{3.15} that
\begin{align*} \nu(G_k)&\le \gamma\Big(\big\{y\in\Bbb R^{u_k}\ |\
\|V_ky-L_ky\|_{\ell_q^{u_k}}>2c_22^{-k\rho}
\lz_{n_k,\sz_k}\big\}\Big)\\ &\le\lz\Big(\big\{y\in\Bbb R^{u_k}\ |\
\|V_ky-L_ky\|_{\ell_q^{u_k}}>2c_22^{-k\rho}
\lz_{n_k,\sz_k}\big\}\Big)\\ &= \gz_{u_k}\Big(\big\{y\in\Bbb
R^{u_k}\ |\ \|V_ky-L_ky\|_{\ell_q^{u_k}}>2
\lz_{n_k,\sz_k}\big\}\Big)\le \sz_k,\end{align*}where $\lz$ is a
centered Gaussian measure in $\Bbb R^{u_k}$ with covariance matrix
$c_2^22^{-2k\rho}I_{u_k}$, $I_{u_k}$ is the identity matrix in $\Bbb
R^{u_k}$. Let us consider the set $G=\bigcup_{k=1}^\infty G_k$ and
the linear operator $\widetilde{T_n}$ on $W_{2,\mu}^r$ which is
given by
$$\widetilde{T_n}f=\sum_{k=1}^\infty T_kR_k^{-1}L_kS_kU_k\dz_k(f).$$ From the hypothesis
of the theorem, we get that
$$\nu(G)\le \sum_{k=1}^\infty \nu(G_k)\le \sum_{k=1}^\infty
\sz_k\le\sz,$$ and
$${\rm rank}\,\widetilde{T_n}\le \sum_{k=1}^\infty {\rm
rank}\,(T_kR_k^{-1}L_kS_kU_k\dz_k)\le \sum_{k=1}^\infty n_k\le
n.$$Consequently, by the definitions of $G,\ \widetilde{T_n},\
\{G_k\}$, and $\{L_k\}$,
\begin{align*}\lz_{n,\dz}\Big(W_{2,\mu}^r,\nu,L_{q,\mu}\Big) &
\le\sup_{f\in W_{2,\mu}^r\backslash
G}\|f-\widetilde{T_n}f\|_{q,\mu}\\&\le \sup_{f\in
W_{2,\mu}^r\backslash G}\sum_{k=1}^\infty
\Big\|\dz_k(f)-T_kR_k^{-1}L_kS_kU_k\dz_k(f)\Big\|_{q,\mu}\\ &\le
\sum_{k=1}^\infty \sup_{f\in W_{2,\mu}^r\backslash
G_k}\Big\|\dz_k(f)-T_kR_k^{-1}L_kS_kU_k\dz_k(f)\Big\|_{q,\mu}\\ &
\ll\sum_{k=1}^\infty 2^{-k\rho}\lz_{n_k,\dz_k},\end{align*}which
completes the proof of Theorem \ref{thm4.3}.\end{proof}

Next we consider the lower estimates. We assume that $m\ge 6$ and
$b_1m^d\le n\le 2b_1m^d$ with $b_1>0$ being independent of $n$ and
$m$. We let $\{x_j\}_{j=1}^{N}\subset \{x\in \ss\,\big|\ \|x\|_2\le
2/3\}$ such that $N\asymp m^d$ and
$$\{x\in \ss\,\big|\ \|x-x_j\|_2\le 2/m\}\bigcap\, \{x\in \ss\,\big|\
\|x-x_i\|_2\le 2/m\}=\emptyset, \ \ \ {\rm  if}\ \ i\neq j.$$
Obviously, such points $x_j$  exist. We may take $b_1>0$
sufficiently large so that $N\ge 2n$. Let  $\vr^1$ be a
$C^\infty$-function on $\R^d$ supported in $\{x\in \R^d\, |\
\|x\|_2\le 1\}$  and be equal to 1 on $\{x\in \R^d\, |\ \|x\|_2\le
2/3\}$, and let $\vr^2$ be  a nonnegative $C^\infty$-function on
$\R^d$ supported in $\{x\in \R^d\, |\ \|x\|_2\le 1/2\}$ and be equal
to 1 on $\{x\in \R^d\, |\ \|x\|_2\le 1/4\}$. We define
$$\vr_i(x)=\vr^1(m(x-x_i))-c_i\vr^2(m(x-x_i)),$$for some $c_i$ such that
$\int_{\ss}\vr_i(x)W_\mu(x)dx=0, \ i=1,\dots,N$.
 We set
$$A_N:={\rm span}\, \{\vr_1,\dots,\vr_{N}\}=\Big\{F_{\bf a}(x)=\sum_{j=1}^{N}a_j\vr_j(x)\,
 :\ {\bf a}=(a_1,\dots,a_{N})\in \Bbb R^N\Big\}.$$
Clearly,  $$\vr_j\in W_{2,\mu}^r,\ \ {\rm supp}\, \vr_j\subset
\{x\in \ss\,\big|\ \|x-x_j\|_2\le 1/m\}\subset\{x\in \ss\,\big|\
\|x\|_2\le 5/6\},$$ $$ \|\vr_j\|_{q,\mu}\asymp
\Big(\int_{\ss}|\vr_j(x)|^q\,dx\Big)^{1/q}\asymp m^{-d/q},\ \ \ 1\le
q\le \infty,\ \ j=1,\dots,N, $$ and
$${\rm supp}\, \vr_j\bigcap {\rm supp}\, \vr_i\ =\emptyset\ \ \
(i\neq j).$$ Hence, if $F_{\bf a}\in A_N,\ {\bf
a}=(a_1,\dots,a_{N})\in \Bbb R^N$, then
\begin{equation}\label{3.16}\|F_{\bf a}\|_{q,\mu}\asymp
\Big(N^{-1}\sum_{j=1}^{N}|a_j|^q\Big)^{1/q}=m^{-d/q}\|{\bf
a}\|_{\ell_q^{N}}. \end{equation}

For a positive integer $v=0,1,\dots$ and $F_{\bf a}\in A_N,\ {\bf
a}=(a_1,\dots,a_{N})\in \Bbb R^N$, it follows from the definition of
$-D_\mu^d$  that
$$\supp(-D_\mu^d)^{v} (\vr_j)\subset \{x\in\ss\ \big|\ \|x-x_j\|_2\le 1/m\},$$ and $$ \|(-D_\mu^d)^{v} (\vr_j)\|_{q,\mu}\ll m^{2v-d/q}.$$
Hence, for $1\le q\le \infty$ and $F_{\bf a}=\sum_{j=1}^Na_j\vr_j\in
A_N$,
$$\|(-D_\mu^d)^{v} (F_{\bf a})\|_{q,\mu}\ll
 m^{2v-d/q}\|{\bf a}\|_{\ell_q^{N}}.$$ It then follows by the
 Kolmogorov type inequality (see \cite[Theorem 8.1]{Di}) that
 \begin{align}\|F_{\bf a}^{(\rho)}\|_{q,\mu}&=\|(-D_\mu^d)^{\rho/2} (F_{\bf a})\|_{q,\mu}\notag\\ &\ll \|(-D_\mu^d)^{1+[\rho]}
 (F_{\bf a})\|_{q,\mu}^{\frac{\rho}{2+2[\rho]}}\ \|F_{\bf a}\|_{q,\mu}^{1-\frac{\rho}{2+2[\rho]}}\notag\\ &\ll
 m^{\rho-d/q}\|{\bf a}\|_{\ell_q^{N}}\ll m^{\rho}\|F_{\bf a}\|_{q,\mu}.\label{3.17}\end{align}

 For $f\in L_{1,\mu}$ and $x\in \ss$,
we define
$$P_N(f)(x)=\sum_{j=1}^N \frac{\vr_j(x)}{\|\vr_j\|_{2,\mu}^2}\int_{\ss}
f(y){\vr_j(y)}W_\mu(y)dy $$ and $$ Q_N(f)(x)=\sum_{j=1}^N
\frac{\vr_j(x)}{\|\vr_j\|_{2,\mu}^2}\int_{\ss} f
(y){\vr_j^{(\rho)}(y)}W_{\mu}(y)dy.$$Obviously, the operator $P_N$
is the orthogonal projector from $L_{2,\mu}$ to $A_N$, and if $f\in
W_{2,\mu}^\rho$, then $Q_N(f)(x)=P_N(f^{(\rho)})(x)$. Also, it
follows from \cite{WH} that $P_{N}$ is the  bounded operator from
$L_{q,\mu}$ to $A_N\bigcap L_{q,\mu}$, i.e., for $1\le q\le \infty$,
\begin{equation}\label{3.18} \|P_{N}(f)\|_{q,\mu}\ll
\|f\|_{q,\mu}.\end{equation}
 Since $Q_N(f)\in A_N$ for $f\in W_{2,\mu}^\rho$,
by \eqref{3.17} we have
\begin{equation}\label{3.19}\|(Q_N(f))^{(\rho)}\|_{2,\mu}\ll m^{\rho} \|Q_N(f)\|_{2,\mu}=m^\rho
\|P_N(f^{(\rho)})\|_{2,\mu}\ll m^\rho
\|f^{(\rho)}\|_{2,\mu}.\end{equation}

\

\begin{thm}\label{thm4.4} Let $1\le q\le \infty,\
 \dz\in(0,1)$, and let $N$ be given above.  Then \begin{equation*}\label{3.20}\lz_{n,\dz}(W_{2,\mu}^r,\nu,L_{q,\mu})\gg
 n^{-\rho/d+1/2-1/q}\lz_{n,\dz}(I_N:\Bbb R^{N}\to \ell_q^N,\gz_{N}), \end{equation*} where $N\asymp n,\ N\ge 2n$, $I_N$
 is the $N$ by $N$ identity matrix, and $\gamma_N$ is   the
standard Gaussian measure in $\Bbb R^N$.\end{thm}
\begin{proof}
 Let $T_n$ be a bounded linear operator on  $W_{2,\mu}^r$ with  rank $T_n\le
n$ such that  $$\nu\big(\{ f\in W_{2,\mu}^r\ \big| \
\|f-T_nf\|_{q,\mu}>2\lz_{n,\delta}\}\big)\le\delta,$$where
$\lz_{n,\delta}:=\lz_{n,\delta}(W_{2,\mu}^r,\nu,L_{q,\mu})$. Note
that if $A$ is a bounded linear operator from $W_{2,\mu}^r$ to
$W_{2,\mu}^r$ and from $H(\nu)$ to $H(\nu)$, then the image measure
$\lz$ of $\nu$ under $A$ is also a centered Gaussian measure on
$W_{2,\mu}^r$ with covariance
$$R_{\lz}(f)(f)=\lb A^*C_\nu f, A^*C_\nu f\rb_{H(\nu)},\ \ \ \
f\in W_{2,\mu}^r,$$where $C_\nu$ is the covariance of the measure
$\nu$, $H(\nu)=W_{2,\mu}^\rho$ is the Camera-Martin space of $\nu$,
and $A^*$ is the adjoint of $A$ in $H(\nu)$ (see \cite[Theorem
3.5.1, p. 112]{Bo}). Furthermore, if the operator $A$ also satisfies
$$\|Af\|_{H(\nu)}\le \|f\|_{H(\nu)},$$then
$$R_\lz(f)(f)=\|A^*C_\nu f\|_{H(\nu)}^2\le \|A^*\|^2\|C_\nu
f\|_{H(\nu)}^2\le \lb C_\nu f, C_\nu f\rb_{H(\nu)}=R_\nu(f)(f).$$ By
Theorem 3.3.6 in \cite[p. 107]{Bo}, we get that for any absolutely
convex Borel set $E$ of $W_{2,\mu}^r$, there holds inequality
$$\nu(E)\le \lz(E).$$ It follows from \eqref{3.19} that
$$\|Q_N(f)\|_{H(\nu)}=\|(Q_N(f))^{(\rho)}\|_{2,\mu}\ll m^\rho\|f^{(\rho)}\|_{2,\mu}=m^{\rho}\|f\|_{H(\nu)}.$$ Then there exists
a positive constant $c_3$ such that $$\|\frac
1{c_3m^\rho}Q_N(f)\|_{H(\nu)}\le \|f\|_{H(\nu)}.$$ Note that for any
$t>0$, the set $\{f\in W_{2,\mu}^r\ |\ \|f-T_nf\|_{q,\mu}\le t\}$ is
absolutely convex. It then follows that
\begin{align*}&\qquad \nu\big(\{ f\in W_{2,\mu}^r\ | \
\|f-T_nf\|_{q,\mu}>2\lz_{n,\delta}\}\big)\\ &\ge \nu\big(\{ f\in
W_{2,\mu}^r\ | \ \|Q_Nf-T_nQ_Nf\|_{q,\mu}>2c_3m^\rho
\lz_{n,\delta}\}\big).\end{align*} Now we define the linear
operators $L_N:\Bbb R^{N}\longmapsto A_N$ and $J_N: A_N\longmapsto
\Bbb R^N$ by
$$L_N({\bf a})(x)=\sum_{i=1}^{N}\frac{a_i\,\vr_i(x)}{\|\vr_i\|_{2,\mu}},\ \ \ {\bf a}=(a_1,\dots,a_N)\in \Bbb R^N$$ and
$$J_N(F_{\bf a})=(a_1\|\vr_1\|_{2,\mu},\dots,a_N\|\vr_N\|_{2,\mu}),  \ \ F_{\bf a}\in A_N,$$
respectively. Obviously, $L_NJ_N(F_{\bf a})=F_{\bf a}$ for any
$F_{\bf a}\in A_N$. Set $y=(y_1,\dots,y_N)\in \Bbb R^N$, where
$y_j=\frac 1{\|\vr_j\|_{2,\mu}}\lb f,\vr_{j}^{(\rho)}\rb$. Then
$y=J_NQ_N(f)$. It follows from \eqref{3.16} and the fact
$\|\vr_j\|_{2,\mu}\asymp m^{-d/2}$ that
\begin{equation}\label{3.21}\|L_N({\bf a})\|_{q,\mu}\asymp
\Big(N^{-1}\sum_{j=1}^{N}\frac{|a_j|^q}{\|\vr_j\|_{2,\mu}^q}\Big)^{1/q}\asymp
m^{-d/q+d/2}\|{\bf a}\|_{\ell_q^{N}}. \end{equation}
  By \eqref{3.18}
and \eqref{3.21}, we know that for any $f\in W_{2,\mu}^r$,
\begin{align*}  \|Q_N(f)-T_nQ_N(f)\|_{q,\mu}&\gg \|P_N(Q_N(f))-P_NT_nQ_N(f)\|_{q,\mu}\\ &=\|L_NJ_NQ_N(f)-L_NJ_NP_NT_nL_NJ_NQ_N(f)\|_{q,\mu}\\ &\gg
 m^{-d/q+d/2}\|J_NQ_N(f)-J_NP_NT_nL_NJ_NQ_N(f)\|_{\ell_q^N}\\
 &\gg m^{-d/q+d/2}\|y-J_NP_NT_nL_Ny\|_{\ell_q^N}.\end{align*}
 We remark that   $g_{k}=\frac{\vr_k}{\|\vr_k\|_{2,\mu}},\ k=1,2,\dots,N$
is an orthonormal  system  in $L_{2,\mu}$ and $g_k\in
H(\nu)=W_{2,\mu}^\rho$. Then the random vector
 $(\lb f,g_1^{(\rho)}\rb,\dots,\lb f,g_N^{(\rho)}\rb)=y$ in $\Bbb R^N$ on the measurable
space $(W_{2,\mu}^r,\nu)$ has the  standard Gaussian distribution
$\gamma_N$ in $\Bbb R^N$. It then follows that
\begin{align*}&\quad \ \nu\big(\{ f\in W_{2,\mu}^r\ \big|
\ \|Q_Nf-T_nQ_Nf\|_{q,\mu}>2c_3m^\rho \lz_{n,\delta}\}\big ) \\ &\ge
\nu\big(\{ f\in W_{2,\mu}^r\ | \ \|y-J_NP_NT_nL_Ny\|_{\ell_q^N}>
c_4m^{\rho+d/q-d/2} \lz_{n,\delta}\}\big)\\ &=\gz_N\big( \{ y\in
\Bbb R^N\ |\ \|y-J_NP_NT_nL_Ny\|_{\ell_q^N}>
c_4m^{\rho+d/q-d/2}\lz_{n,\delta}\} \big)=:\gz_N(G),\end{align*}
where $c_4$ is a positive constant. Clearly, rank$\,(
J_NP_NT_nL_N)\le n$ and
$$ \gz_N(G)\le \nu\big(\{ f\in W_{2,\mu}^r\ \big| \
\|f-T_nf\|_{q,\mu}>2\lz_{n,\delta}\}\big)\le \delta.$$ Therefore,
\begin{align*}\lz_{n,\delta}(I_N:\Bbb R^N\to\ell_q^N, \gz_N)&\le \sup_{y\in\Bbb R^N\backslash G} \|y-J_NP_NT_nL_Ny\|_{\ell_q^N}
\ll m^{\rho+d/q-d/2}\lz_{n,\delta}.\end{align*} That is,
$$\lz_{n,\delta}(W_{2,\mu}^r,\nu,L_{q,\mu})\gg n^{-\rho/d+1/2-1/q}
\lz_{n,\delta}(I_N:\Bbb R^N\to\ell_q^N, \gz_N),$$which completes the
proof of Theorem \ref{thm4.4}.
\end{proof}

\section{Proofs of Theorems \ref{thm1.1} and \ref{thm1.2}}

\noindent{\it Proof of Theorem \ref{thm1.1}.}\ \  The lower
estimates for $\lz_{n,\dz}(W_{2,\mu}^r,\nu,L_{q,\mu})$ follow from
Theorem \ref{thm4.4}, and \eqref{2.1},   \eqref{2.2},   and
\eqref{2.3} for $1\le q\le \infty$ immediately.

For the upper estimates for $\lz_{n,\dz}(W_{2,\mu}^r,\nu,L_{q,\mu})$
for $2\le q\le \infty$, we use Theorem \ref{thm4.3}. For any fixed
natural number $n$, assume $C_12^{md}\le n\le C_1^2 2^{md}$ with
$C_1>0$ to be specified later.
 We may take sufficiently
  small  positive  numbers $\vz>0$ such that $\rho>\frac d2 +(1+\vz)(2+\vz)\mu d(\frac12-\frac1q)$ and define
\begin{equation*}n_j=\left\{\begin{array}{ll} u_j,  & {\rm if}\ \  j\le m, \\
\Big[ u_j2^{d(1+\vz)(m-j)-1}\Big], & {\rm if}\   j> m
,\end{array}\right.  \  {\rm and} \ \dz_j=
\left\{\begin{array}{ll} 0,  & {\rm if}\  j\le m, \\
\dz 2^{m-j}, & {\rm if}\   j> m
,\end{array}\right.\end{equation*}where $u_j$ is given  in Theorem
\ref{thm4.3}. Then
$$\sum_{j\ge0}n_j\ll \sum_{j\le m}2^{jd}+
\sum_{j>m}2^{md(1+\vz)-d\vz j}\ll 2^{md},
$$ and $$\sum_{k=1}^\infty \dz_k\le \dz.$$Hence, we can take $C_1$ sufficiently large so that
$\sum_{j=0}^\infty n_j\le C_12^{md}\le n$. It follows from Lemma
\ref{lem4.2} that for $\beta\in (0,\frac1{2\mu(1/2-1/q)}),\ 2\le
q\le \infty$, $$\sum_{j=1}^{u_k}\oz_j^{-\beta(1/2-1/q)}\ll
2^{kd}2^{kd\beta(\frac12-\frac1q)}.$$If $j\le m$, then $n_j=u_j$,
and thence $\lz_{n_j,\dz_j}(V_j:\Bbb R^{u_j}\to \ell_q^{u_j},
\gz_{u_j})=0$. If $j>m$,  then taking $\frac1
\beta=(2+\vz)\mu(\frac12-\frac1q)$  and applying Lemma \ref{2.2}, we
obtain for $2\le q<\infty$,
\begin{equation}\label{5.1}\lz_{n_j,\dz_j}(V_j:\RR^{u_j}\to \ell_q^{u_j},\gz_{u_j})\ll 2^{jd(\frac12-\frac1q)-d(1+\vz)(m-j)(2+\vz)\mu(\frac12-\frac1q)}(2^{\frac{jd}q}+(\ln(\frac1\dz))^{\frac12}) ,
\end{equation}
and for $q=\infty$,
\begin{equation}\label{5.2}\lz_{n_j,\dz_j}(V_j:\RR^{u_j}\to \ell_q^{u_j},\gz_{u_j})\ll 2^{{jd}/2-d(1+\vz)(m-j)(2+\vz)\mu/2}
\sqrt{j+\ln(1/\dz)}.
\end{equation}
 Now we estimate the upper bounds for
$\lz_{n,\dz}(W_2^r,\mu,L_q)$ for $1\le q\le \infty$. For $2\le
q<\infty$, by \eqref{3.13} and \eqref{5.1} we get
\begin{align*}  &\quad\ \lz_{n,\dz}(W_{2,\mu}^r,\nu,L_{q,\mu})\\
&\ll \sum_{j=m+1}^{\infty}2^{-j(\rho-\frac
d2+\frac dq)}2^{-d(1+\vz)(m-j)(2+\vz)\mu(\frac12-\frac1q)}(2^{\frac{jd}q}+(\ln(\frac1\dz))^{\frac12})\\
&\ll 2^{-m(\rho-\frac d2+\frac dq)}\Big(2^{md/q}+
(\ln(1/\dz))^{1/2}\Big)\\
& \ll n^{-\rho/d+1/2}\big(1+n^{-2/q}\ln (1/\dz)\big)^{1/2} .
\end{align*}
For $q=\infty$, it follows from \eqref{3.13} and \eqref{5.2} that
\begin{align*} \lz_{n,\dz}(W_{2,\mu}^r,\nu,L_{\infty,\mu})
&\ll \sum_{j=m+1}^{\infty}2^{-j\rho}
2^{{jd}/2-d(1+\vz)(m-j)(2+\vz)\mu/2} \sqrt{j+\ln(1/\dz)}\\ &\ll
2^{-m(\rho-d/2)}\sqrt{m+\ln(1/\dz)}\\ & \ll n^{-\rho/d+1/2}\sqrt
{\ln (n/\dz)} . \end{align*}
 For $1\le q<2$, we have
$$\lz_{n,\dz}(W_{2,\mu}^r,\nu,L_{q,\mu})\le \lz_{n,\dz}(W_{2,\mu}^r,\nu,L_{2,\mu})\ll n^{-\rho/d+1/2}\big(1+n^{-1}\ln (1/\dz)\big)^{1/2}.$$
 The proof of Theorem \ref{thm1.1} is
complete.$\hfill\Box$

\

\noindent{\it Proof of Theorem \ref{thm1.2}.}\ \    By the
definition of $\lz_{n,\dz}(W_{2,\mu}^r,\nu,L_{q,\mu})$,  there
exists a linear operator $L_n$ with ${\rm rank}\le n$ such that for
any $\dz\in(0,1/2]$ and some subset $G_\dz\subset W_{2,\mu}^r$ with
$\nu(G_\dz)\le \dz$,
$$\sup_{f\in W_{2,\mu}^r\backslash G_\dz}\|f-L_nf\|_{q,\mu}\le 2\lz_{n,\dz}(W_{2,\mu}^r,\nu,L_{q,\mu}).$$
Consider the sequence $\{G_{2^{-k}}\}_{k=0}^\infty$
 of sets, where $G_1=W_{2,\mu}^r$. Then it follows from \eqref{1.6}
 that \begin{align*}\lz_n^{(a)}(W_{2,\mu}^r,\nu,L_{q,\mu})_p&\le\Big( \int_{W_{2,\mu}^r}\|f-L_nf\|_{q,\mu}^p\,\nu(df)\Big)^{1/p}\\
  &=\Big(\sum_{k=0}^\infty\int_{G_{2^{-k}}\backslash G_{2^{-k-1}}}\|f-L_nf\|_{q,\mu}^p\,\nu(df)\Big)^{1/p}\\
  &\le \Big(\sum_{k=0}^\infty(2\lz_{n,2^{-k-1}}(W_{2,\mu}^r,\nu,L_{q,\mu}))^p\,\nu(G_{2^{-k}})\Big)^{1/p}\\ &\ll
  \Big(\sum_{k=0}^\infty2^{-k}(\lz_{n,2^{-k-1}}(W_{2,\mu}^r,\nu,L_{q,\mu}))^p\Big)^{1/p}\\ &\ll \bigg\{\begin{array}{ll}
n^{-\rho/d+1/2},\ \ \ \ &1\le q<\infty,\\
n^{-\rho/d+1/2}\sqrt{\ln (en)}, &q=\infty.\end{array}\end{align*}

For the proof of the lower estimates for
$\lz_n^{(a)}(W_{2,\mu}^r,\nu,L_{q,\mu})_p$, let $L_n$ be a linear
operator from $L_{q,\mu}$ to $L_{q,\mu}$ with rank at most $n$. We
set
$$G'=\{f\in W_{2,\mu}^r\ \big|\ \|f-L_nf\|_{q,\mu}\ge \frac
12\lz_{n,1/e}(W_{2,\mu}^r,\nu,L_{q,\mu})\}.$$ Then $\nu (G')\ge
1/e$. Otherwise, if $\nu(G')<1/e$, then by the definition of
probabilistic linear $(n,\dz)$-width, we have
$$\lz_{n,1/e}(W_{2,\mu}^r,\nu,L_{q,\mu})\le \sup_{f\in W_{2,\mu}^r\backslash
G'}\|f-L_nf\|_{q,\mu}\le \frac
12\lz_{n,1/e}(W_{2,\mu}^r,\nu,L_{q,\mu}),$$which leads to a
contradiction. Hence $\nu(G')\ge 1/e$. It follows from \eqref{1.6}
that for $0<p<\infty$,
\begin{align*}\Big(\int_{W_{2,\mu}^r}\|f-L_nf\|_{q,\mu}^p\,\nu(df)\Big)^{1/p}&\ge
\Big(\int_{G'}\|f-L_nf\|_{q,\mu}^p\,\nu(df)\Big)^{1/p}\\ &\ge \frac
12\lz_{n,1/e}(W_{2,\mu}^r,\nu,L_{q,\mu})\,(\nu(G'))^{1/p}\\ &\gg
\bigg\{\begin{array}{ll}
n^{-\rho/d+1/2},\ \ \ \ &1\le q<\infty,\\
n^{-\rho/d+1/2}\sqrt{\ln (en)},
&q=\infty,\end{array}\end{align*}which gives the required  lower
estimates for $\lz_n^{(a)}(W_{2,\mu}^r,\nu,L_{q,\mu})_p$. This
completes the proof of Theorem \ref{thm1.2}. $\hfill\Box$

\end{document}